\documentclass[a4paper]{amsart}
\usepackage[leqno]{amsmath}
\usepackage{enumerate}
\usepackage{hyphenat}
\usepackage{amstext,amsbsy,amsopn,amsthm,amsgen}
\usepackage{amsfonts,amscd,amsxtra,upref}
\usepackage{mathrsfs}
\usepackage{euscript}
\usepackage{amssymb}
\usepackage{hyperref}

\swapnumbers
\theoremstyle{plain}
\newtheorem{Thm}{Theorem}[section]
\newtheorem{Lem}[Thm]{Lemma}
\newtheorem{Cor}[Thm]{Corollary}
\newtheorem{Pro}[Thm]{Proposition}
\theoremstyle{definition}
\newtheorem{Def}[Thm]{Definition}
\theoremstyle{remark}
\newtheorem{Rem}[Thm]{Remark}

\numberwithin{equation}{section}

\newenvironment{cor}[1]{\begin{Cor}\label{cor:#1}}{\end{Cor}}
\newenvironment{dfn}[1]{\begin{Def}\label{def:#1}}{\end{Def}}
\newenvironment{lem}[1]{\begin{Lem}\label{lem:#1}}{\end{Lem}}
\newenvironment{pro}[1]{\begin{Pro}\label{pro:#1}}{\end{Pro}}
\newenvironment{rem}[1]{\begin{Rem}\label{rem:#1}}{\end{Rem}}
\newenvironment{thm}[1]{\begin{Thm}\label{thm:#1}}{\end{Thm}}

\newcommand{\COR}[1]{Corollary~\textup{\ref{cor:#1}}}
\newcommand{\DEF}[1]{Definition~\textup{\ref{def:#1}}}
\newcommand{\LEM}[1]{Lemma~\textup{\ref{lem:#1}}}
\newcommand{\PRO}[1]{Proposition~\textup{\ref{pro:#1}}}
\newcommand{\THM}[1]{Theorem~\textup{\ref{thm:#1}}}

\newcommand{\CCC}{\mathbb{C}}\newcommand{\RRR}{\mathbb{R}}\newcommand{\TTT}{\mathbb{T}}
\newcommand{\BBb}{\CMcal{B}}\newcommand{\EEe}{\CMcal{E}}\newcommand{\GGg}{\CMcal{G}}
\newcommand{\HHh}{\CMcal{H}}\newcommand{\UUu}{\CMcal{U}}\newcommand{\WWw}{\CMcal{W}}
\newcommand{\FfF}{\EuScript{F}}\newcommand{\IiI}{\EuScript{I}}\newcommand{\JjJ}{\EuScript{J}}
\newcommand{\MmM}{\EuScript{M}}\newcommand{\NnN}{\EuScript{N}}\newcommand{\WwW}{\EuScript{W}}
\newcommand{\Bb}{\mathfrak{B}}\newcommand{\Dd}{\mathfrak{D}}\newcommand{\Ff}{\mathfrak{F}}
\newcommand{\Gg}{\mathfrak{G}}\newcommand{\Ii}{\mathfrak{I}}\newcommand{\Mm}{\mathfrak{M}}
\newcommand{\Ss}{\mathfrak{S}}\newcommand{\bB}{\mathfrak{b}}
\newcommand{\aaA}{\mathscr{A}}\newcommand{\bbB}{\mathscr{B}}\newcommand{\ffF}{\mathscr{F}}
\newcommand{\iiI}{\mathscr{I}}\newcommand{\kkK}{\mathscr{K}}\newcommand{\mmM}{\mathscr{M}}
\newcommand{\ssS}{\mathscr{S}}\newcommand{\uuU}{\mathscr{U}}\newcommand{\xxX}{\mathscr{X}}
\newcommand{\zzZ}{\mathscr{Z}}
\newcommand{\sSS}{\pmb{S}}\newcommand{\tTT}{\pmb{T}}
\newcommand{\ueA}{\textup{\textsf{A}}}\newcommand{\ueB}{\textup{\textsf{B}}}
\newcommand{\ueJ}{\textup{\textsf{J}}}\newcommand{\ueS}{\textup{\textsf{S}}}
\newcommand{\ueT}{\textup{\textsf{T}}}\newcommand{\ueX}{\textup{\textsf{X}}}
\newcommand{\ueY}{\textup{\textsf{Y}}}

\newcommand{\dd}{\colon}\newcommand{\df}{\stackrel{\textup{def}}{=}}\newcommand{\geqsl}{\geqslant}
\newcommand{\leqsl}{\leqslant}\newcommand{\epsi}{\varepsilon}\newcommand{\dint}[1]{\,\textup{d} #1}
\newcommand{\Card}{\operatorname{Card}}\newcommand{\id}{\operatorname{id}}
\newcommand{\tI}{\textup{I}}\newcommand{\disj}{\perp_u}\newcommand{\sdisj}{\perp_s}
\newcommand{\contS}[3][]{\left.\bigoplus_{#2}#1\!\right.^{#3}}
\begin{document}

\title{Direct integrals of matrices}
\author[P. Niemiec]{Piotr Niemiec}
\address{Instytut Matematyki\\{}Wydzia\l{} Matematyki i~Informatyki\\{}%%
 Uniwersytet Jagiello\'{n}ski\\{}ul. \L{}ojasiewicza 6\\{}30-348 Krak\'{o}w\\{}Poland}
\email{piotr.niemiec@uj.edu.pl}
\thanks{The author gratefully acknowledges the assistance of the Polish Ministry of Sciences
 and Higher Education grant NN201~546438 for the years 2010--2013.}
\begin{abstract}
It is shown that each linear operator on a separable Hilbert space which generates a finite type~I
von Neumann algebra has, up to unitary equivalence, a unique representation as a direct integral
of inflations of mutually unitary inequivalent irreducible matrices. This leads to a simplification
of the so-called prime (or central) decomposition and the multiplicity theory for such operators.
The concept of so-called \textit{p-isomorphisms} between special classes of such operators is
discussed. All results are formulated in more general settings; that is, for tuples of closed
densely defined operators affiliated with finite type~I von Neumann algebras.
\end{abstract}
\subjclass[2010]{Primary 47B40; Secondary 47C15.}
\keywords{Direct integral; finite type~I von Neumann algebra; central decomposition of an operator;
 prime decomposition of a tuple of operators; spectral theorem; multiplicity theory.}
\maketitle

\section{Introduction}

In \cite{ern} Ernest proved that every bounded linear operator acting on a separable Hilbert space
has, up to unitary equivalence, a unique so-called \textit{central decomposition}. His definition
of this decomposition involved the reduction theory of von Neumann algebras (due to von Neumann
himself \cite{neu}). More precisely, the direct integral
\begin{equation}\label{eqn:di}
T = \int^{\oplus} T_{\omega} \dint{\mu(\omega)}
\end{equation}
is the central decomposition of $T$ iff $\int^{\oplus} \WwW(T_{\omega}) \dint{\mu(\omega)}$ is
the central decomposition of $\WwW(T)$ where for any bounded linear operator $S$, $\WwW(S)$ denotes
the smallest von Neumann algebra containing $S$. Recently \cite{pn1} we generalised Ernest's theorem
to the context of finite tuples of closed densely defined linear operators acting in a common
(totally arbitrary, possibly nonseparable) Hilbert space. We call our result the Prime Decomposition
Theorem (see Chapter~5 in \cite{pn1} and Theorem~5.6.14 there). In practice it is not so easy
to resolve whether the direct integral of the form \eqref{eqn:di} is the central decomposition
of a separable Hilbert space operator $T$. A necessary condition for this is that off some null set,
$\WwW(T_{\omega})$ is a factor and the operators $T_{\omega}$ and $T_{\omega'}$ for distinct
$\omega$ and $\omega'$ have no nontrivial unitary equivalent suboperators (in that case we call
the operators $T_{\omega}$ and $T_{\omega'}$ \textit{unitarily disjoint}). The main disadvantage
of both the results cited above is that the last mentioned condition is insufficient in general for
\eqref{eqn:di} to be the central decomposition of $T$. So, some bounded operators $T$ admit two
inequivalent representations as direct integrals of pairwise unitarily disjoint factor operators.
However, as we will show in this paper, if $T$ generates a finite type~I von Neumann algebra, its
central decomposition may simply be recognised: in that case, a necessary and sufficient condition
for \eqref{eqn:di} to be the central decomposition of $T$ is that off some null set, the field
$\{T_{\omega}\}$ consists of inflations of pairwise unitarily inequivalent irreducible matrices,
as shown by

\begin{thm}{0}
\begin{enumerate}[\upshape(a)]
\item For $j \in \{1,2\}$, let $\{T_j(\omega_j)\}_{\omega_j\in\Omega_j}$ be a measurable field
 defined on a standard measure space $(\Omega_j,\Mm_j,\mu_j)$ such that off some $\mu_j$-null set,
 the operators $T_j(\omega_j)$ are unitarily equivalent to inflations of pairwise unitarily
 inequivalent irreducible matrices. Then
 \begin{equation}\label{eqn:ue}
 \int^{\oplus}_{\Omega_1} T_1(\omega_1) \dint{\mu_1(\omega_1)} \textup{ and }
 \int^{\oplus}_{\Omega_2} T_2(\omega_2) \dint{\mu_2(\omega_2)} \textup{ are unitarily equivalent}
 \end{equation}
 iff there are measurable sets $Z_1 \subset \Omega_1$ and $Z_2 \subset \Omega_2$, and a Borel
 isomorphism $\Phi\dd \Omega_1 \setminus Z_1 \to \Omega_2 \setminus Z_2$ such that:
 \begin{itemize}
 \item $\mu_1(Z_1) = \mu_2(Z_2) = 0$;
 \item for any measurable set $B \subset \Omega_1 \setminus Z_1$, $\mu_1(B) = 0$ if and only if
  $\mu_2(\Phi(B)) = 0$;
 \item the operators $T_2(\Phi(\omega_1))$ and $T_1(\omega_1)$ are unitarily equivalent for any
  $\omega_1 \in \Omega_1 \setminus Z_1$.
 \end{itemize}
\item Let $T$ be a closed densely defined linear operator in a separable Hilbert space. Then $T$ is
 unitarily equivalent to a decomposable operator of the form
 \begin{equation}\label{eqn:repr}
 \int^{\oplus}_{\Omega} T(\omega) \dint{\mu(\omega)}
 \end{equation}
 where $\{T(\omega)\}_{\omega\in\Omega}$ is a measurable field defined on a standard measure space
 $(\Omega,\Mm,\mu)$ such that the operators $T(\omega)$ are unitarily equivalent to inflations
 of pairwise unitarily inequivalent irreducible matrices iff $T$ is affiliated with a finite type~I
 von Neumann algebra.
\end{enumerate}
\end{thm}

A certain generalisation of the above theorem (which will be stated and proved in Section~3) shall
be applied in the sequel to prove that certain classes of tuples of operators are
\textit{p-isomorphic}. Roughly speaking, the existence of a p-isomorphism between two classes
asserts that these classes have the same Borel (as well as spectral) complexity (p-isomorphisms
preserve direct sums of arbitrary collections and prime decompositions). As a main result in this
direction we will obtain a theorem stating that the class of all $k$-tuples (where $k$ is
arbitrarily fixed) of closed densely defined linear operators affiliated with finite type~I
von Neumann algebras is p-isomorphic to the class of all (single) unitary operators.\par
The paper is organised as follows. In Section~2 we study finite Borel measures on sets of tuples
of matrices. The main property established there is used in Section~3 to prove \THM{0} and its
certain generalisation to finite tuples of operators (\THM{pr_deco} below). The last, fourth, part
is devoted to so-called \textit{finite type~I ideals} of unitary equivalence classes of tuples
of operators. We recall there our Prime Decomposition Theorem from \cite{pn1} and prove that some
finite type~I ideals are p-isomorphic. We conclude the paper with a concept of multiplicity theory
for finite tuples of operators which are affiliated with finite type~I von Neumann algebras.

\subsection*{Notation and terminology} All Hilbert spaces (as well as matrices) are complex.
An \textit{operator} means a closed linear operator densely defined in a Hilbert space. Each
$k$-tuple of operators (resp. matrices) consists of operators acting in a common, but arbitrary,
Hilbert space (resp. of square matrices whose degrees coincide). For a $k$-tuple $(T_1,\ldots,T_k)$
of bounded operators, we denote by $\WwW(T_1,\ldots,T_k)$ the smallest von Neumann algebra which
contains each of $T_1,\ldots,T_k$. If $T_1,\ldots,T_k$ are merely closed and densely defined, let
$\WwW'(T_1,\ldots,T_k)$ stand for the von Neumann algebra generated by all unitary operators $U$ for
which $U T_j U^{-1} = T_j\ (j=1,\ldots,k)$. Additionally, we put $\WwW''(T_1,\ldots,T_k) =
(\WwW'(T_1,\ldots,T_k))'$. Recall that $\WwW''(T_1,\ldots,T_k) = \WwW(T_1,\ldots,T_k)$ provided
$T_1,\ldots,T_k$ are bounded. The $k$-tuple $(T_1,\ldots,T_k)$ is said to be \textit{irreducible}
if $\WwW'(T_1,\ldots,T_k)$ consists (precisely) of scalar multiples of the identity operator. Two
$k$-tuples $(T_1,\ldots,T_k)$ and $(T_1',\ldots,T_k')$ are \textit{unitarily equivalent} if there
exists a single unitary operator $U$ (between respective Hilbert spaces) such that $U T_s U^{-1} =
T_s'$ for each $s \in \{1,\ldots,k\}$. By an \textit{inflation} of $(T_1,\ldots,T_k)$ we mean
a $k$-tuple of the form $(\bigoplus_{s \in S} T_1^{(s)},\ldots,\bigoplus_{s \in S} T_k^{(s)})$ where
$S$ is a nonempty set and $T_j^{(s)} = T_j$ for any $s \in S$ and $j \in \{1,\ldots,k\}$ (with this
terminology we follow e.g. Conway \cite{con} --- see Example~5.3 there). The $\bB$-transform (cf.
\cite{pn1}) of an operator $T$ is given by $\bB(T) \df T(I+|T|)^{-1}$ where $|T| =
(T^* T)^{\frac12}$ is the absolute value of $T$ and $I$ is the identity operator on a respective
Hilbert space. Whenever $X$ is a topological space, $\Bb(X)$ denotes the $\sigma$-algebra of all
Borel sets in $X$; that is, $\Bb(X)$ is the smallest $\sigma$-algebra which contains all open
subsets of $X$. A \textit{map} is a continuous function.\par
For an easier understanding of the paper, it is recommended to study treatises \cite{pn1} and
\cite{ern}.

\section{Measures on sets of tuples of matrices}

For each $n \geqsl 1$ denote by $\mmM_n$ the $C^*$-algebra of all $n \times n$ matrices and
by $\uuU_n$ the unitary group of $\mmM_n$ (with unit $I$). Fix $k \geqsl 1$ (as the length
of tuples) and equip the vector space $\mmM_n^k$ of all $k$-tuples of $n \times n$ matrices with
the action of $\uuU_n$ given by
\begin{equation*}
U.(A_1,\ldots,A_k) = (U A_1 U^{-1},\ldots,U A_k U^{-1}) \qquad (U \in \uuU_n,\ (A_1,\ldots,A_k) \in
\mmM_n^k).
\end{equation*}
Since our main interest are irreducible $k$-tuples, the following simple lemma may be helpful
in understanding the notion introduced above. We leave its proof to the reader (use the fact that
$\WwW'(A_1,\ldots,A_k)$ is a von Neumann algebra).

\begin{lem}{irr}
A tuple $(A_1,\ldots,A_k) \in \mmM_n^k$ is irreducible iff for any $U \in \uuU_n$,
\begin{equation*}
U.(A_1,\ldots,A_k) = (A_1,\ldots,A_k) \iff U = \gamma \cdot I \textup{ for some } \gamma \in \CCC.
\end{equation*}
\end{lem}

We use $\iiI_n(k)$ to denote the set of all irreducible $k$-tuples of $n \times n$ matrices. It is
well-known that $\iiI_n(k)$ is an open set in $\mmM_n^k$ (hence $\iiI_n(k)$ is a separable locally
compact space). Observe also that $\iiI_n(k)$ is invariant under the action of $\uuU_n$. We now
introduce

\begin{dfn}{uecs}
Let $\Ii_n(k)$ stand for the orbit space of the action of $\uuU_n$ on $\iiI_n(k)$; that is,
$\Ii_n(k)$ consists of all sets of the form $\{U.(A_1,\ldots,A_k)\dd\ U \in \uuU_n\}$ where
$(A_1,\ldots,A_k) \in \iiI_n(k)$. Further, let $\pi_n^k\dd \iiI_n(k) \to \Ii_n(k)$ be the natural
projection. We equip $\Ii_n(k)$ with the quotient topology.
\end{dfn}

Since the group $\uuU_n$ is compact and $\iiI_n(k)$ is locally compact, it is easily seen that
$\Ii_n(k)$ is a separable metrizable locally compact space and $\pi_n^k$ is a proper map. What is
more,

\begin{pro}{select}
There exists a $\GGg_{\delta}$-set $\ssS_n(k) \subset \iiI_n(k)$ which meets the orbit of each
irreducible $k$-tuple at exactly one point; in other words, the restriction of $\pi_n^k$
to $\ssS_n(k)$ is a Borel isomorphism of $\ssS_n(k)$ onto $\Ii_n(k)$.
\end{pro}

The above result may be deduced e.g. from Corollary~1 in \S2 of Chapter~XIV in \cite{k-m}, applied
to the partition $\Ii_n(k)$ of $\iiI_n(k)$, or from Corollary~1 in \S1 of Chapter~XIV in \cite{k-m}
(see also \cite{cas}; for more information on selection theorems consult Chapter~XIV of \cite{k-m}
or \S12.D of Chapter~II in \cite{kec}).\par
To simplify further statements, we fix $k > 0$ and for any $n > 0$ a set $\ssS_n(k) \subset
\iiI_n(k)$ which witnesses the assertion of \PRO{select}, and consider the topological disjoint
union $\ssS(k) = \bigsqcup_{n=1}^{\infty} \ssS_n(k)$ of the topological spaces $\ssS_n(k)$, $n > 0$.
Note that $\ssS(k)$, as a topological disjoint union of $\GGg_{\delta}$-sets, is separable and
completely metrizable. The aim of this section is to show that all $\sigma$-finite Borel measures
on $\ssS(k)$ are supported on certain sets, which we now introduce.

\begin{dfn}{mdsud}
A set $\ffF$ of $k$-tuples of matrices is said to be a \textit{measurable domain of strong unitary
disjointness} (or, briefly, a \textit{measurable domain}) if $\ffF \cap \mmM_n^k$ is Borel for any
$n > 0$ and there is a countable collection $\{F_1,F_2,\ldots\}$ of $k$-tuples of matrices which
separates points of $\ffF$ (that is, for any two distinct tuples of $\ffF$ there is a number $s$
such that $F_s$ contains exactly one of them) and for any number $m$ there exists a sequence
$p_1,p_2,\ldots$ of complex polynomials in $2k$ noncommuting variables such that the matrices
$p_n(\bB(X_1),\ldots,\bB(X_k),\bB(X_1)^*,\ldots,\bB(X_k)^*)$ converge to $0$ for any
$(X_1,\ldots,X_k) \in \ffF \setminus F_m$ and to a respective unit matrix for any $(X_1,\ldots,X_k)
\in \ffF \cap F_m$.
\end{dfn}

Measurable domains were introduced in \cite{pn1} and used to characterise so-called \textit{prime
decompositions} of tuples of operators. A more detailed explanation of this will be given
in \THM{prime} (consult also \THM{cov} in Section~4).\par
The aim of this part is to prove

\begin{thm}{meas}
Every $\sigma$-finite Borel measure on $\ssS(k)$ is supported on a measurable domain.
\end{thm}

The proof of the above theorem shall be preceded by a few auxiliary results. The first of them is
a kind of folklore. Since we could not find it in the literature, we sketch its simple proof.

\begin{lem}{cantor}
If $\mu$ is a finite Borel measure on a separable complete metric space $X$ and $B$ is a Borel set
in $X$, then for every $\epsi > 0$ there is a totally disconnected compact set $K \subset B$ such
that $\mu(B \setminus K) \leqsl \epsi$.
\end{lem}
\begin{proof}
It is well-known that finite Borel measures on separable complete metric spaces are
\textit{regular}; that is, $\mu(C) = \sup \{\mu(L)\dd\ L \subset C,\ L \textup{ compact}\}$ for each
Borel set $C$. Since for any $\delta > 0$ the set $B$ may be covered by a countable number
of pairwise disjoint Borel sets whose diameters are less than $\delta$, we infer from the regularity
of $\mu$ that for any $s > 0$, there is a finite collection $\{K_1^{(s)},\ldots,K_{p_s}^{(s)}\}$
of pairwise disjoint compact subsets of $B$ whose diameters are less than $\frac1s$ such that
$\mu(B \setminus \bigcup_{n=1}^{p_s} K_n^{(s)}) \leqsl \frac{\epsi}{2^s}$. Then the set $K =
\bigcap_{s=1}^{\infty} \bigl(\bigcup_{n=1}^{p_s} K_n^{(s)}\bigr)$ is totally disconnected and
compact, and satisfies $\mu(B \setminus K) \leqsl \epsi$.
\end{proof}

The following result may be seen as a special case of the factorial Stone\hyp{}Weierstrass theorem
due to Longo \cite{lon} and Popa \cite{pop}.

\begin{lem}{s-w}
Let $\kkK$ be a compact subset of $\ssS(k) \cap \mmM_n^k$ for some $n > 0$. Then the $C^*$-algebra
$C(\kkK,\mmM_n)$ of all $\mmM_n$-valued continuous functions on $\kkK$ coincides with the smallest
unital $C^*$-subalgebra of $C(\kkK,\mmM_n)$ which contains the maps $b_j\dd \kkK \ni
(X_1,\ldots,X_k) \mapsto \bB(X_j) \in \mmM_n\ (j=1,\ldots,k)$.
\end{lem}
\begin{proof}
Denote by $\EEe$ a unital $C^*$-algebra generated by $\{b_1,\ldots,b_k\}$. Let $(X_1,\ldots,X_k)$
and $(Y_1,\ldots,Y_k)$ be two distinct members of $\kkK$. Since $(X_1,\ldots,X_k)$ is irreducible,
so is $(\bB(X_1),\ldots,\bB(X_k))$ and, consequently,
\begin{equation}\label{eqn:aux1}
\{f(X_1,\ldots,X_k)\dd\ f \in \EEe\} = \mmM_n.
\end{equation}
Furthermore, since
\begin{equation*}
\pi_n^k(\bB(X_1),\ldots,\bB(X_k)) \neq \pi_n^k(\bB(Y_1),\ldots,\bB(Y_k)),
\end{equation*}
the irreducibility of the $k$-tuples $(\bB(X_1),\ldots,\bB(X_k))$ and $(\bB(Y_1),\ldots,\bB(Y_k))$
is followed by the fact that there is a complex polynomial $p$ in $2k$ noncommuting variables such
that
\begin{gather*}
p(\bB(X_1),\ldots,\bB(X_k),\bB(X_1)^*,\ldots,\bB(X_k)^*) = I,\\
p(\bB(Y_1),\ldots,\bB(Y_k),\bB(Y_1)^*,\ldots,\bB(Y_k)^*) = 0
\end{gather*}
(cf. e.g. \cite{pn2}). So, for some $g \in \EEe$,
\begin{equation}\label{eqn:aux2}
g(X_1,\ldots,X_k) = I \qquad \textup{and} \qquad g(Y_1,\ldots,Y_k) = 0.
\end{equation}
Now a combination of \eqref{eqn:aux1} and \eqref{eqn:aux2} implies that $\EEe = C^*(\kkK,\mmM_n)$,
which follows directly from Proposition~2.3 in \cite{pn2} or Corollary~11.5.3 in \cite{dix}.
\end{proof}

As a consequence, we obtain

\begin{cor}{tot_disc}
Let $\kkK$ be a totally disconnected and compact subset of $\ssS(k) \cap \mmM_n^k$ for some $n > 0$.
Then $\kkK$ is a measurable domain.
\end{cor}
\begin{proof}
Let $\FfF$ be a collection of all clopen (that is, simultaneously open and closed) subsets
of $\kkK$. Then $\FfF$ is countable (because $\kkK$ is metrizable and compact), separates points
of $\kkK$ (by the total disconnectedness of $\kkK$) and, by \LEM{s-w}, for any $F \in \FfF$,
the function $j_F\dd \kkK \to \mmM_n$ given by
\begin{equation*}
j_F(X_1,\ldots,X_k) = \begin{cases}I & (X_1,\ldots,X_k) \in F,\\
0 & (X_1,\ldots,X_k) \notin F\end{cases}
\end{equation*}
belongs to the smallest unital $C^*$-algebra which contains the maps $b_1,\ldots,b_k$ introduced
in \LEM{s-w}. Consequently, for any $F \in \FfF$, there is a sequence $p_1,p_2,\ldots$
of complex polynomials in $2k$ noncommuting variables such that
\begin{equation*}
\lim_{n\to\infty} p_n(\bB(X_1),\ldots,\bB(X_k),\bB(X_1)^*,\ldots,\bB(X_k)^*) = j_F(X_1,\ldots,X_k)
\end{equation*}
for any $(X_1,\ldots,X_k) \in \kkK$ (and the convergence is uniform on $\kkK$). Now the very
definition of a measurable domain yields the assertion.
\end{proof}

\begin{proof}[Proof of \THM{meas}]
For the purpose of this proof, we follow some concepts of \cite{pn1}. In particular:
\begin{itemize}
\item we call a $\sigma$-finite Borel measure on $\ssS(k)$ a \textit{regularity} measure if it is
 supported on a measurable domain;
\item for two regularity measures $\mu$ and $\nu$ we write $\mu \sdisj \nu$ if $\mu$ and $\nu$ are
 mutually singular and $\mu+\nu$ is a regularity measure as well;
\item two $k$-tuples $(T_1,\ldots,T_k)$ and $(T_1',\ldots,T_k')$ of operators are \textit{unitarily
 disjoint}, in symbols $(T_1,\ldots,T_k) \disj (T_1',\ldots,T_k')$, if no nontrivial part
 of $(T_1,\ldots,T_k)$ is unitarily equivalent to a part of $(T_1',\ldots,T_k')$.
\end{itemize}
So, our aim is to show that in the above context the property of being a regularity measure is
automatic.\par
Let $\mu$ be a $\sigma$-finite Borel measure on $\ssS(k)$. We may and do assume that $\mu$ is
finite. It follows from \LEM{cantor} that there is a collection $\{\kkK_{n,m}\dd\ n,m > 0\}$
of mutually disjoint totally disconnected compact sets such that for any $n > 0$:
\begin{itemize}
\item $\kkK_{n,m} \subset \ssS(k) \cap \mmM_n^k$ for any $m > 0$;
\item $\mu(\ssS(k) \cap \mmM_n^k \setminus \bigcup_{m=1}^{\infty} \kkK_{n,m}) = 0$.
\end{itemize}
For simplicity, let $\mu_{n,k}$ be a Borel measure given by $\mu_{n,k}(\bbB) = \mu(\bbB \cap
\kkK_{n,m})$. It follows from \COR{tot_disc} that $\mu_{n,k}$ is a regularity measure. We shall now
show that
\begin{equation}\label{eqn:sdisj}
\mu_{n,m} \sdisj \mu_{n',m'}
\end{equation}
for any two distinct pairs $(n,m)$ and $(n',m')$. To this end, we will separately deal with
the cases when $n = n'$ and $n \neq n'$. In the former case, \eqref{eqn:sdisj} follows immediately
from \COR{tot_disc} (because $K_{n,m} \cup K_{n,m'}$ is totally disconnected). Now assume $n \neq
n'$ and consider $k$-tuples $(T_1,\ldots,T_k)$ and $(T_1',\ldots,T_k')$ with $T_j \df
\int^{\oplus}_{\kkK} X_j \dint{\mu_{n,m}(X_1,\ldots,X_k)}$ and $T_j' \df \int^{\oplus}_{\kkK} X_j
\dint{\mu_{n,m'}(X_1,\ldots,X_k)}$. In other words (following \cite{pn1}),
\begin{equation}\label{eqn:p_deco}
(T_1,\ldots,T_k) = \int^{\oplus}_{\kkK} \id \dint{\mu_{n,m}} \qquad \textup{and} \qquad
(T_1',\ldots,T_k') = \int^{\oplus}_{\kkK} \id \dint{\mu_{n,m'}}
\end{equation}
(where `$\id$' stands for the identity map). Since both $\mu_{n,m}$ and $\mu_{n,m'}$ are regularity
measures supported on, respectively, $\ssS(k) \cap \mmM_n^k$ and $\ssS(k) \cap \mmM_{n'}^k$,
it follows e.g. from Corollary~5.6.7 in \cite{pn1} that $\WwW(T_1,\ldots,T_k)$ and
$\WwW(T_1',\ldots,T_k')$ are, respectively, type~$\tI_n$ and $\tI_{n'}$ von Neumann algebras
(because \eqref{eqn:p_deco} are \textit{prime decompositions}; consult \cite{pn1} or Theorem~4.6
in Section~4 and the notes following it). This implies that $(T_1,\ldots,T_k) \disj
(T_1',\ldots,T_k')$ and thus \eqref{eqn:sdisj} holds, thanks to Lemma~5.4.10 in \cite{pn1}.\par
Having \eqref{eqn:sdisj}, one easily deduces from Lemma~5.6.12 (still in \cite{pn1}) that $\mu$,
which coincides with $\sum_{n=1}^{\infty} \sum_{m=1}^{\infty} \mu_{n,m}$, is a regularity measure.
\end{proof}

\section{Prime decomposition in finite type~I algebras}

Let $A$ and $m$ be an $n \times n$ matrix and a positive integer, respectively. We denote by $m
\odot A$ the inflation of $A$ composed of $m$ copies of $A$. That is,
\begin{equation*}
m \odot A = \begin{pmatrix}A & 0 & \cdots & 0\\
 0 & A & \cdots & 0\\
 \vdots & \vdots & \ddots & \vdots\\
 0 & 0 & \cdots & A\end{pmatrix}
\end{equation*}
is a square matrix of degree $mn$. Additionally, for simplicity, we will write $\infty \odot A$
to denote the inflation of $A$ (acting on $\ell_2$) composed of countable infinite number of copies
of $A$. In a similar manner we denote inflations of tuples of operators. We shall say a field
$\{(X_1(\omega),\ldots,X_k(\omega))\}_{\omega\in\Omega}$ of $k$-tuples is \textit{measurable}
if each of the fields $\{X_s(\omega)\}_{\omega\in\Omega}\ (s=1,\ldots,k)$ of operators is measurable
(cf. \cite{pn1}).\par
It is an easy exercise and a well-known fact that if $\{X(\omega)\}_{\omega\in\Omega}$ and $f\dd
\Omega \to \{1,2,\ldots,\infty\}$ are, respectively, a measurable field of matrices and a measurable
function, then the field $\{f(\omega) \odot X(\omega)\}_{\omega\in\Omega}$ is measurable as well.
What we want to prove is the following

\begin{thm}{pr_deco}
For $j \in \{1,2\}$, let $\{(X_1^{(j)}(\omega_j),\ldots,
X_k^{(j)}(\omega_j))\}_{\omega_j\in\Omega_j}$ and $f_j\dd \Omega_j \to \{1,2,\ldots,\infty\}$ be,
respectively, a measurable field of $k$-tuples of matrices and a measurable function defined
on a standard measure space $(\Omega_j,\Mm_j,\mu_j)$ such that:
\begin{enumerate}[\upshape(F1)]
\item for any $\omega_j \in \Omega_j$, the $k$-tuple
 $(X_1^{(j)}(\omega_j),\ldots,X_k^{(j)}(\omega_j))$ is irreducible;
\item for any two distinct points $\omega_j$ and $\omega_j'$ of $\Omega_j$, the $k$-tuples
 \begin{equation*}
 (X_1^{(j)}(\omega_j),\ldots,X_k^{(j)}(\omega_j)) \qquad \textup{and} \qquad
 (X_1^{(j)}(\omega_j'),\ldots,X_k^{(j)}(\omega_j'))
 \end{equation*}
 are unitarily inequivalent.
\end{enumerate}
Then the $k$-tuples
\begin{equation*}
\tTT_1 \df \Bigl(\int^{\oplus}_{\Omega_1} f_1(\omega_1) \odot X_1^{(1)}(\omega_1)
\dint{\mu_1(\omega_1)},\ldots,\int^{\oplus}_{\Omega_1} f_1(\omega_1) \odot X_k^{(1)}(\omega_1)
\dint{\mu_1(\omega_1)}\Bigr)
\end{equation*}
and
\begin{equation*}
\tTT_2 \df \Bigl(\int^{\oplus}_{\Omega_2} f_2(\omega_2) \odot X_1^{(2)}(\omega_2)
\dint{\mu_2(\omega_2)},\ldots,\int^{\oplus}_{\Omega_2} f_2(\omega_2) \odot X_k^{(2)}(\omega_2)
\dint{\mu_2(\omega_2)}\Bigr)
\end{equation*}
are unitarily equivalent iff there exist measurable sets $Z_1 \subset \Omega_1$, $Z_2 \subset
\Omega_2$ and a Borel isomorphism $\Phi\dd \Omega_1 \setminus Z_1 \to \Omega_2 \setminus Z_2$ such
that:
\begin{enumerate}[\upshape(E1)]
\item $\mu_1(Z_1) = 0$ and $\mu_2(Z_2) = 0$;
\item for each measurable set $B \subset \Omega_1 \setminus Z_1$, $\mu_2(\Phi(B)) = 0$ if and only
 if $\mu_1(B) = 0$;
\item the $k$-tuples $(X_1^{(2)}(\Phi(\omega_1)),\ldots,X_k^{(2)}(\Phi(\omega_1)))$ and
 $(X_1^{(1)}(\omega_1),\ldots,X_k^{(1)}(\omega_1))$ are unitarily equivalent and
 $f_2(\Phi(\omega_1)) = f_1(\omega_1)$ for any $\omega_1 \in \Omega_1 \setminus Z_1$.
\end{enumerate}
\end{thm}
\begin{proof}
The `if' part is easy and well-known (and holds in more general settings; see e.g. items
(di5)--(di6) on page~73 in \cite{pn1}). Here we shall focus only on the `only if' part. So, assume
$\tTT_1$ and $\tTT_2$ are unitarily equivalent. Below we preserve the settings and notation
of the previous section. Further, let $\iiI(k)$ denote $\bigcup_{n=1}^{\infty} \iiI_n(k)$ and
$\psi\dd \iiI(k) \to \ssS(k)$ be a function given by the rule: $\psi(A_1,\ldots,A_k)$ is a unique
$k$-tuple in $\ssS(k)$ which is unitarily equivalent to $(A_1,\ldots,A_k)$. Then $\psi$ is a Borel
function. Now for $j \in \{1,2\}$, let $\varphi_j\dd \Omega_j \to \ssS(k)$ be given
by $\varphi_j(\omega_j) = \psi(X_1^{(j)}(\omega_j),\ldots,X_k^{(j)}(\omega_j))$. We see that
$\varphi_j$ is measurable. What is more, it is one-to-one, by (F2). Since $(\Omega_j,\Mm_j,\mu_j)$
is a standard Borel space, we conclude from \THM{meas} (involving e.g. the transport of $\mu_j$
under $\varphi_j$) that there is a measurable set $Z_j' \subset \Omega_j$ such that $\mu_j(Z_j') =
0$ and $\varphi_j(\Omega_j \setminus Z_j')$ is a measurable domain. This implies that the set
$\{(X_1^{(j)}(\omega_j),\ldots,X_k^{(j)}(\omega_j))\dd\ \omega_j \in \Omega_j \setminus Z_j'\}$ is
a measurable domain as well. This means that the field
$\{(X_1^{(j)}(\omega_j),\ldots,X_k^{(j)}(\omega_j))\}_{\omega_j\in\Omega_j}$ is \textit{regular}
in the sense of \cite{pn1} (cf. Proposition~5.4.4 there) and thus also
\begin{equation}\label{eqn:reg}
\{(f_j(\omega_j) \odot X_1^{(j)}(\omega_j),\ldots,f_j(\omega_j) \odot
X_k^{(j)}(\omega_j))\}_{\omega_j\in\Omega_j} \textup{ is a regular field}
\end{equation}
(see Lemma~5.4.8 in \cite{pn1}). To make further arguments more transparent, we need to recall some
concepts introduced in \cite{pn1}. We consider the set $\Ff$ of all unitary equivalence classes
of (all) $k$-tuples of separable Hilbert space operators, equipped with a quotient Borel structure,
as described in \cite{pn1}. For simplicity, we denote by $\Phi_j\dd \Omega_j \to \Ff$ a (measurable)
function defined by the rule: $\Phi_j(\omega_j)$ is the unitary equivalence class of $(f_j(\omega_j)
\odot X_1^{(j)}(\omega_j),\ldots,f_j(\omega_j) \odot X_k^{(j)}(\omega_j))$. It follows from
\eqref{eqn:reg} and the choice of $Z_j'$ that the set $\Phi_j(\Omega_j \setminus Z_j')$ is Borel
in $\Ff$ and the restriction of $\Phi_j$ to $\Omega_j \setminus Z_j'$ is a Borel isomorphism; and
from item (F2) that $\Phi_j$ is one-to-one (on the whole $\Omega_j$). Now let $\lambda_j$ be
the transport of $\mu_j$ under $\Phi_j$; that is, $\lambda_j(\Gg) = \mu_j(\Phi_j^{-1}(\Gg))$ for any
Borel set $\Gg \subset \Ff$. Since $\tTT_1$ and $\tTT_2$ are unitarily equivalent, we infer from
Corollary~5.4.7 in \cite{pn1} that the measures $\lambda_1$ and $\lambda_2$ are mutually absolutely
continuous. So, the intersection $\Ff_0$ of $\Phi_1(\Omega_1 \setminus Z_1')$ and $\Phi_2(\Omega_2
\setminus Z_2')$ is of full measure with respect to both $\lambda_1$ and $\lambda_2$. It suffices
to put $Z_j = \Omega_j \setminus \Phi_j^{-1}(\Ff_0) (\supset Z_j')$ and define $\Phi\dd \Omega_1
\setminus Z_1 \to \Omega_2 \setminus Z_2$ as $(\Phi_2\bigr|_{\Omega_2 \setminus Z_2})^{-1} \circ
\Phi_1$. It is readily seen that $\Phi$ is a Borel isomorphism and conditions (E1)--(E2) are
fulfilled. Finally, observe that for any $\omega_1 \in \Omega_1 \setminus Z_1$,
$\Phi_2(\Phi(\omega_1)) = \Phi_1(\omega_1)$ which means that the $k$-tuples $(f_2(\Phi(\omega_1))
\odot X_1^{(2)}(\Phi(\omega_1)),\ldots,f_2(\Phi(\omega_1)) \odot X_k^{(2)}(\Phi(\omega_1)))$ and
$(f_1(\omega_1) \odot X_1^{(1)}(\omega_1),\ldots,f_1(\omega_1) \odot X_k^{(1)}(\omega_1))$ are
unitarily equivalent. So, an application of (F1) yields that (E3) holds and we are done.
\end{proof}

The proof of \THM{0} depends on the property established below. The next result is almost certainly
known, but we could not find it in the literature. (A similar, but slightly different, property was
established by Ernest in Lemma~4.3 in \cite{ern}.) For reader's convenience, we give its proof.

\begin{lem}{infl}
If $\{(T_1(\omega),\ldots,T_k(\omega))\}_{\omega\in\Omega}$ is a measurable field defined
on a standard Borel space $(\Omega,\Mm)$ such that the $k$-tuples $(T_1(\omega),\ldots,T_k(\omega))$
are unitarily equivalent to inflations of irreducible $k$-tuples of matrices, then there exist
a measurable function $f\dd \Omega \to \{1,2,\ldots,\infty\}$ and a measurable field
$\{(A_1(\omega),\ldots,A_k(\omega))\}_{\omega\in\Omega}$ of irreducible $k$-tuples of matrices
such that the $k$-tuples $(T_1(\omega),\ldots,T_k(\omega))$ and $f(\omega) \odot (A_1(\omega),
\ldots,A_k(\omega))$ are unitarily equivalent for each $\omega \in \Omega$.
\end{lem}
\begin{proof}
The set $\Omega$ may be divided into countably many measurable parts, say $\Omega_1,\Omega_2,\ldots,
\Omega_{\infty}$, such that $T_s(\omega)\ (s=1,\ldots,k)$ acts on a separable Hilbert space $\HHh_n$
of dimension $n$ for any $\omega \in \Omega_n$. Further, we may and do assume that $\HHh_n = \CCC^n$
for finite $n$ and $\HHh_{\infty} = \ell_2$. We now fix $n \in \{1,2,\ldots,\infty\}$ and put $J(n)
\df \{p \in \{1,2,\ldots\}\dd\ n/p \in \{1,2,\ldots,\infty\}\}$. Everywhere below the algebra
$\BBb(\HHh_n)$ of all bounded operators on $\HHh_n$ is equipped with the strong operator topology.
Further, let $\ssS_1(k),\ssS_2(k),\ldots$ be sets obtained from \PRO{select}. For any $p \in J(n)$
we use $\UUu_p$ to denote the set of all unitary operators on $\HHh_n$ which commute with any
operator of the form $\frac{n}{p} \odot X_s$ for $s \in \{1,\ldots,k\}$ where $(X_1,\ldots,X_k) \in
\ssS_p(k)$. Then $\UUu_p$ is a closed subgroup of the unitary group of $\HHh_n$. By a theorem
of Dixmier (see e.g. Theorem~12.17 in Chapter~II of \cite{kec} or Theorem~1.2.4 in \cite{b-k}),
there exists a Borel set $\EEe_p$ of unitary operators on $\HHh_n$ which meets every left coset
of $\UUu_p$ (in the whole unitary group) at exactly one point. Since $\ssS_p(k)$ consists
of mutually unitarily inequivalent irreducible $k$-tuples, we conclude that the function $h_n\dd
\bigcup_{p \in J(n)} \EEe_p \times \ssS_p(k) \to \BBb(\HHh_n)^k$ (where $\BBb(\HHh_n)$ is
the algebra of all bounded operators on $\HHh_n$ equipped with the strong operator topology) given
by $h_n(U,(A_1,\ldots,A_k)) = (U (\frac{n}{p} \odot A_1) U^{-1},\ldots,U (\frac{n}{p} \odot A_k)
U^{-1})$ (for $U \in \EEe_p$, $(A_1,\ldots,A_k) \in \ssS_p(k)$) is one-to-one. Furthermore, for any
$\omega \in \Omega_n$ there is a unique pair $(U(\omega),(A_1(\omega),\ldots,A_k(\omega))) \in
\bigcup_{p \in J(n)} \EEe_p \times \ssS_p(k)$ for which $(T_1(\omega),\ldots,T_k(\omega)) =
h_n(U(\omega),(A_1(\omega),\ldots,A_k(\omega)))$. Since $h_n$ is one-to-one, it is a Borel
isomorphisms and thus the field $\{(A_1(\omega),\ldots,A_k(\omega))\}_{\omega\in\Omega_n}$ is
measurable. Now it suffices to put $f(\omega) \df \frac{n}{p}$ provided $\omega \in \Omega_n$ is
such that $(A_1(\omega),\ldots,A_k(\omega)) \in \ssS_p(k)$. The very definition of $f$ shows that
$f$ is measurable and the whole construction yields the assertion of the lemma.
\end{proof}

Below we give a proof of \THM{0} in more general settings; that is, for $k$-tuples in place
of single operators.

\begin{proof}[Proof of \THM{0}]
Assume all the assumptions of item (a) are fulfilled. Then it follows from \LEM{infl} that we may
find measurable fields $\{(A_j(\omega_j)\}_{\omega_j\in\Omega_j}\ (j=1,2)$ of irreducible matrices
and measurable functions $f_j\dd \Omega_j \to \{1,2,\ldots,\infty\}\ (j=1,2)$ such that
$T_j(\omega_j)$ is unitarily equivalent to $f_j(\omega_j) \odot A_j(\omega_j)$ for $\mu_j$-almost
all $\omega_j \in \Omega_j$. We conclude that then $\int^{\oplus}_{\Omega_1} f_1(\omega_1) \odot
A_1(\omega_1) \dint{\mu_1(\omega_1)}$ is unitarily equivalent to $\int^{\oplus}_{\Omega_2}
f_2(\omega_2) \odot A_2(\omega_2) \dint{\mu_2(\omega_2)}$. Now an application of \THM{pr_deco}
easily finishes the proof of that part.\par
We turn to (b). First of all, observe that $T$ is affiliated with a finite type~I von Neumann
algebra iff $\WWw''(T)$ is such an algebra (because $T$ is affiliated with $\MmM$ iff $\MmM' \subset
\WwW'(T)$, iff $\WwW''(T) \subset \MmM$; and being finite and type~I is a hereditary property for
von Neumann algebras). So, if $\WwW''(T)$ is finite and type~I, the conclusion follows from Ernest's
central decomposition \cite{ern} or our Prime Decomposition Theorem \cite{pn1}. Conversely, if $T$
is unitarily equivalent to \eqref{eqn:repr} where $\{T(\omega)\}$ and $(\Omega,\Mm,\mu)$ are
as in item (b) of the theorem, then we argue as follows. \LEM{infl} shows we may replace $T(\omega)$
by $f(\omega) \odot A(\omega)$ where $A(\omega)$'s are mutually unitarily inequivalent irreducible
matrices. Further, the proof of \THM{pr_deco} includes the information that then the field
$\{f(\omega) \odot A(\omega)\}_{\omega\in\Omega}$ is regular and hence $\int^{\oplus}_{\Omega}
f(\omega) \odot A(\omega) \dint{\mu(\omega)}$ is the prime decomposition of $T$ (for details,
consult \cite{pn1}). Hence $\WwW''(T)$ is finite and type~I, because $\WwW''(A(\omega))$ is so for
any $\omega \in \Omega$ (see Corollary~5.6.7 in \cite{pn1}).
\end{proof}

\begin{rem}{tuple}
The reasoning presented above proves also a counterpart of \THM{0} for $k$-tuples of operators,
whose precise formulation and detailed proof are left to the reader as an exercise.
\end{rem}

\section{Finite type~I ideals}

First we recall the concept of an ideal introduced in \cite{pn1}.

\begin{dfn}{ideal}
A class $\IiI$ of $k$-tuples of operators is said to be an \textit{ideal} if $\IiI$ is nonempty and
each of the following three conditions is fulfilled:
\begin{enumerate}[({I}D1)]\addtocounter{enumi}{-1}
\item if $(T_1,\ldots,T_k) \in \IiI$ and a $k$-tuple $(T_1',\ldots,T_k')$ (acting in a totally
 arbitrary Hilbert space) is unitarily equivalent to $(T_1,\ldots,T_k)$, then $(T_1',\ldots,T_k')
 \in \IiI$;
\item if $\{(T_1^{(s)},\ldots,T_k^{(s)})\}_{s \in S}$ is a nonempty collection of $k$-tuples
 belonging to $\IiI$, then $(\bigoplus_{s \in S} T_1^{(s)},\ldots,\bigoplus_{s \in S} T_k^{(s)}) \in
 \IiI$;
\item whenever $(T_1,\ldots,T_k)$ is a member of $\IiI$ acting in a Hilbert space $\HHh$ and $\EEe$
 is a closed linear subspace of $\HHh$ (including trivial) which reduces each of $T_1,\ldots,T_k$,
 then $(T_1\bigr|_{\EEe},\ldots,T_k\bigr|_{\EEe})$ belongs to $\IiI$ as well.
\end{enumerate}
\end{dfn}

Ideals are natural generalisations (to $k$-tuples of operators) of part-properties studied
by Brown, Fong and Hadwin in \cite{bfh}.\par
For simplicity, we extend the notation (for inflations) introduced in the previous section
to arbitrary cardinal numbers: whenever $\alpha$ is a positive cardinal number and
$(X_1,\ldots,X_k)$ is a $k$-tuple of operators acting in a Hilbert space $\HHh$, then $\alpha \odot
(X_1,\ldots,X_k)$ denotes the inflation of $(X_1,\ldots,X_k)$ consisting of $\alpha$ copies of this
tuple; more formally, $\alpha \odot (X_1,\ldots,X_k)$ acts in $\bigoplus_{\xi<\bar{\alpha}} \HHh$
and coincides with $(\bigoplus_{\xi<\bar{\alpha}} X_1,\ldots,\bigoplus_{\xi<\bar{\alpha}} X_k)$
where $\bar{\alpha}$ is the initial ordinal of cardinality $\alpha$.\par
In the sequel we shall need the following characterisation of ideals, formulated as Corollary~3.6.6
in \cite{pn1}.

\begin{lem}{id}
A nonempty class $\IiI$ of $k$-tuples of operators which satisfies condition \textup{(ID0)} is
an ideal iff the following two conditions are fulfilled:
\begin{enumerate}[\upshape({I}D1')]
\item whenever $(T_1^{(s)},\ldots,T_k^{(s)})_{s \in S} \subset \IiI$ is a nonempty collection
 of mutually unitarily disjoint $k$-tuples of operators acting in separable Hilbert spaces, then
 $(\bigoplus_{s \in S} T_1^{(s)},\ldots,\bigoplus_{s \in S} T_k^{(s)}) \in \IiI$;
\item if $(X_1,\ldots,X_k)$ and $(Y_1,\ldots,Y_k)$ are two unitarily disjoint $k$-tuples such that
 $(X_1 \oplus Y_1,\ldots,X_k \oplus Y_k) \in \IiI$, then $(X_1,\ldots,X_k) \in \IiI$;
\item for any $k$-tuple $(X_1,\ldots,X_k)$ and each positive cardinal number $\alpha$, $\alpha \odot
 (X_1,\ldots,X_k) \in \IiI$ iff $(X_1,\ldots,X_k) \in \IiI$.
\end{enumerate}
\end{lem}

\begin{lem}{fin1}
Let $\IiI^f_I$ be the class of all $k$-tuples $(T_1,\ldots,T_k)$ of operators \textup{(}including
all those acting in zero-dimensional Hilbert spaces\textup{)} such that $\WwW''(T_1,\ldots,T_k)$ is
finite and type~I. Then $\IiI^f_I$ is an ideal.
\end{lem}
\begin{proof}
Let $\IiI_I$ and $\IiI^f$ be the classes of all $k$-tuples $(T_1,\ldots,T_k)$ (including all those
acting in zero-dimensional Hilbert spaces) such that $\WwW''(T_1,\ldots,T_k)$ is type~I and,
respectively, finite. Then $\IiI^f_I = \IiI^f \cap \IiI_I$, and $\IiI_I$ is an ideal, by \cite{pn1}.
Therefore it is enough to show that $\IiI^f$ is an ideal. It is clear that (ID0) is satisfied. If
$\{(T_1^{(s)},\ldots,T_k^{(s)})\}_{s \in S}$ is a nonempty collection of members of $\IiI^f$, then
the algebra $\WwW''(\bigoplus_{s \in S} T_1^{(s)},\ldots,\bigoplus_{s \in S} T_k^{(s)})$ is
(naturally) $*$-isomorphic to a unital subalgebra of $\prod_{s \in S} \WwW''(T_1^{(s)},\ldots,
T_k^{(s)})$ and thus $\WwW''(\bigoplus_{s \in S} T_1^{(s)},\ldots,\bigoplus_{s \in S} T_k^{(s)})$ is
finite. This shows (ID1) as well as (ID1'). Further, since $\WwW''(\alpha \odot (X_1,\ldots,X_k))$
is $*$-isomorphic to $\WwW''(X_1,\ldots,X_k)$, we see that also (ID3') holds. So, it remains
to check (ID2'). To this end, assume $(X_1,\ldots,X_k)$ and $(Y_1,\ldots,Y_k)$ are two unitarily
disjoint $k$-tuples such that $(X_1 \oplus Y_1,\ldots,X_k \oplus Y_k) \in \IiI^f$. We infer from
the unitary disjointness of them that $\WwW''(X_1 \oplus Y_1,\ldots,X_k \oplus Y_k)$ is
$*$-isomorphic to $\WwW''(X_1,\ldots,X_k) \times \WwW''(Y_1,\ldots,Y_k)$ which implies that both
the von Neumann algebras $\WwW''(X_1,\ldots,X_k)$ and $\WwW''(Y_1,\ldots,Y_k)$ are finite. This
completes the proof.
\end{proof}

We now introduce

\begin{dfn}{fin1}
An ideal $\IiI$ of $k$-tuples is said to be \textit{finite and type~I} if $\IiI \subset \IiI^f_I$.
\end{dfn}

To every ideal of $k$-tuples there naturally corresponds an ideal of unitary equivalence classes
of $k$-tuples (based on (ID0)). We shall identify both these concepts (i.e., in the realms
of `concrete' $k$-tuples and their unitary equivalence classes). In \cite{pn1} we formulated Prime
Decomposition Theorem for totally arbitrary ideals of unitary equivalence classes of $k$-tuples.
It involved so-called \textit{coverings} defined on (so-called) \textit{multi-standard measurable
spaces with nullities} which we now recall. A \textit{measurable space with nullity} is a triple
$(\xxX,\Mm,\NnN)$ where $\xxX$ is a set, $\Mm$ a $\sigma$-algebra of subsets of $\xxX$ and $\NnN$ is
a $\sigma$-ideal in $\Mm$. $\NnN$ is a counterpart of the $\sigma$-ideal of all sets of measure zero
in a measure space and thus its members are called \textit{null} sets. A set $\bbB \in \Mm$
in a measurable space with nullity $(\xxX,\Mm,\NnN)$ is called \textit{standard} if there exists
a $\sigma$-finite (or, equivalently, finite) positive measure $\mu$ on $\Mm\bigr|_{\bbB} \df \{\aaA
\in \Mm\dd\ \aaA \subset \bbB\}$ such that the $\sigma$-ideals $\NnN(\mu) \df \{\zzZ \in
\Mm\bigr|_{\bbB}\dd\ \mu(\zzZ) = 0\}$ and $\NnN\bigr|_{\bbB} \df \{\zzZ \in \NnN\dd\ \zzZ \subset
\bbB\}$ coincide, and $(\bbB,\Mm\bigr|_{\bbB},\mu)$ is a standard measure space. Finally,
$(\xxX,\Mm,\NnN)$ is \textit{multi-standard} if there exists a collection (called a \textit{standard
base}) $\{\bbB_s\}_{s \in S} \subset \Mm$ of pairwise disjoint sets satisfying the following
conditions:
\begin{itemize}
\item $\bbB_s$ is standard for any $s \in S$;
\item $\xxX \setminus \bigcup_{s \in S} \bbB_s \in \NnN$;
\item a subset $\aaA$ of $\bigcup_{s \in S} \bbB_s$ belongs to $\Mm$ iff $\aaA \cap \bbB_s \in \Mm$
 for each $s \in S$;
\item an arbitrary set $\aaA \subset \bigcup_{s \in S} \bbB_s$ belongs to $\NnN$ iff $\aaA \cap
 \bbB_s \in \NnN$ for all $s \in S$.
\end{itemize}
Next, we recall so-called `continuous' direct sums of measurable fields defined on multi-standard
measurable spaces with nullities. To make the presentation shorter, here we restrict only
to measurable fields whose values are irreducible $k$-tuples of matrices. To this end, let
$\Ii_n(k)$ be as in \DEF{uecs} and let $\Ii(k) = \bigsqcup_{n=1}^{\infty} \Ii_n(k)$ be
the topological disjoint union of all $\Ii_n(k)$'s (with fixed $k$). Then $\Ii(k)$ is separable,
locally compact and metrizable. Further, let $\ssS(k)$ be as in Section~2 (that is, $\ssS(k)$ meets
the unitary equivalence class of each irreducible $k$-tuple of matrices at exactly one point). Now
let $(\xxX,\Mm,\NnN)$ be a multi-standard measurable space with nullity, $\{\bbB_s\}_{s \in S}$ be
its standard base and for each $s \in S$, let $\mu_s$ be a standard measure on $\Mm\bigr|_{\bbB_s}$
such that $\NnN(\mu_s) = \NnN\bigr|_{\bbB_s}$. A cardinal-valued function $f\dd \xxX \to \Card$ is
said to be \textit{measurable} if:
\begin{itemize}
\item $f^{-1}(\{\alpha\}) \in \Mm$ for any cardinal $\alpha$;
\item for any standard set $\bbB \in \Mm$ there is $\zzZ \in \NnN$ such that the set $f(\bbB
 \setminus \zzZ)$ is countable (finite or not).
\end{itemize}
Additionally, $f$ is \textit{nontrivial} if $\xxX \setminus f^{-1}(\{0\}) \notin \NnN$. For
a measurable function $\Phi\dd \xxX \to \Ii(k)$ and a measurable nontrivial cardinal-valued function
$f\dd \xxX \to \Card$ one defines (following \cite{pn1}) $\bigoplus^{\NnN}_{x\in\xxX} f(x) \odot
\Phi(x)$ as the unitary equivalence class of the $k$-tuple $(X_1,\ldots,X_k)$ where
\begin{equation*}
X_j = \bigoplus_{\alpha>0} \alpha \odot \Bigl(\bigoplus_{s \in S} \int^{\oplus}_{\bbB_{s,\alpha}}
T_j(x) \dint{\mu_s(x)}\Bigr)
\end{equation*}
with $\bbB_{s,\alpha} = \bbB_s \cap f^{-1}(\{\alpha\})$ and for each $x \in \xxX$, $(T_1(x),\ldots,
T_k(x))$ is a (unique) member of $\ssS(k)$ whose unitary equivalence class coincides with $\Phi(x)$.
One proves (consult \cite{pn1}) that the above definition is correct; that is, it is independent
of the choice of the transversal $\ssS(k)$, standard base $\{\bbB_s\}_{s \in S}$ and standard
measures $\mu_s$'s.\par
Finally, a \textit{covering} is any measurable field $\Phi\dd \xxX \to \Ii(k)$ defined
on a multi-standard measurable space with nullity $(\xxX,\Mm,\NnN)$ such that
$\bigoplus^{\NnN}_{x\in\bbB} \Phi(x)$ is unitarily disjoint from $\bigoplus^{\NnN}_{x\in\bbB'}
\Phi(x)$ for any two disjoint sets $\bbB, \bbB' \in \Mm$. Coverings were characterised
in Theorem~5.5.5 in \cite{pn1}. Below we formulate its equivalent form for coverings with values
in $\Ii(k)$. Below $\psi\dd \Ii(k) \to \ssS(k)$ is a function given by the rule: for each $\Dd \in
\Ii(k)$, $\psi(\Dd)$ is a unique point of $\Dd \cap \ssS(k)$.

\begin{thm}{cov}
A measurable function $\Phi\dd \xxX \to \Ii(k)$ defined on a multi-standard measurable space with
nullity $(\xxX,\Mm,\NnN)$ is a covering iff for any standard set $\bbB \in \Mm$ there is $\zzZ \in
\NnN$ such that $\psi(\Phi(\bbB \setminus \zzZ))$ is a measurable domain and the restriction
of $\psi \circ \Phi$ to $\bbB \setminus \zzZ$ is a Borel isomorphism \textup{(}of $\bbB \setminus
\zzZ$ onto its image\textup{)}.
\end{thm}

Prime Decomposition Theorem (Theorem~5.6.14 in \cite{pn1}) for finite type~I ideals reads
as follows.

\begin{thm}{prime}
Let $\IiI$ be a finite type~I ideal of unitary equivalence classes of $k$-tuples.
\begin{enumerate}[\upshape(A)]
\item There exists a covering $\Phi\dd \xxX \to \Ii(k)$ defined on $(\xxX,\Mm,\NnN)$ such that
 \begin{equation*}
 \IiI = \IiI(\Phi) \df \Bigl\{\contS[\!]{x\in\xxX}{\NnN} f(x) \odot \Phi(x)|\quad f\dd \xxX \to
 \Card \textup{ measurable}\Bigr\}.
 \end{equation*}
\item If $\Phi\dd \xxX \to \Ii(k)$ and $\Phi\dd \xxX' \to \Ii(k)$ are coverings defined on,
 respectively, $(\xxX,\Mm,\NnN)$ and $(\xxX',\Mm',\NnN')$ such that $\IiI(\Phi) = \IiI(\Phi')$, then
 there are sets $\zzZ \in \NnN$, $\zzZ' \in \NnN'$ and a Borel isomorphism $\Psi\dd \xxX \setminus
 \zzZ \to \xxX' \setminus \zzZ'$ such that $\{\Psi(\bbB)\dd\ \bbB \in \NnN\bigr|_{\zzZ}\} =
 \NnN'\bigr|_{\zzZ'}$ and $\Phi'(\Psi(x)) = \Phi'(x)$ for any $x \in \xxX \setminus \xxX'$.
\item If $\Phi\dd \xxX \to \Ii(k)$ is a covering as in \textup{(A)}, then for two measurable
 functions $f,g\dd \xxX \to \Card$ one has:
 \begin{itemize}
 \item $\bigoplus^{\NnN}_{x\in\xxX} f(x) \odot \Phi(x) = \bigoplus^{\NnN}_{x\in\xxX} g(x) \odot
  \Phi(x)$ iff $f = g$ a.e.; that is, iff $\{x \in \xxX\dd\ f(x) \neq g(x)\} \in \NnN$;
 \item $\bigoplus^{\NnN}_{x\in\xxX} (f(x)+g(x)) \odot \Phi(x) = (\bigoplus^{\NnN}_{x\in\xxX} f(x)
  \odot \Phi(x)) \oplus (\bigoplus^{\NnN}_{x\in\xxX} g(x) \odot \Phi(x))$.
 \end{itemize}
\item If $\Phi\dd \xxX \to \Ii(k)$ is a covering, then $\IiI(\Phi)$ is a finite type~I ideal.
\end{enumerate}
\end{thm}

The above result establishes a one-to-one correspondence between finite type~I ideals (of unitary
equivalence classes of $k$-tuples) and, up to isomorphism (in the sense of item (B) above),
coverings taking values in $\Ii(k)$. Furthermore, if $\IiI$ is any such ideal and $\Phi\dd \xxX \to
\Ii(k)$ is its covering on $(\xxX,\Mm,\NnN)$, then for every $\ueX \in \IiI$ there exists a unique
(up to $\NnN$-almost everywhere equality) cardinal-valued measurable function $m_{\ueX}\dd \xxX \to
\Card$ such that
\begin{equation}\label{eqn:pr_deco}
\ueX = \contS[\!]{x\in\xxX}{\NnN} m_{\ueX}(x) \odot \Phi(x).
\end{equation}
The above representation is called the \textit{prime decomposition} of $\ueX$ and $m_{\ueX}$ is said
to be the \textit{multiplicity} function of $\ueX$ (both in $\IiI$), see \cite{pn1}.\par
So, any characterisation of coverings helps understanding the complexity of ideals. Below we give
a simple criterion for an $\Ii(k)$-valued function to be a covering.

\begin{pro}{cov}
A measurable function $\Phi\dd \xxX \to \Ii(k)$ defined on $(\xxX,\Mm,\NnN)$ is a covering iff
for every standard set $\bbB \in \Mm$ there is $\zzZ \in \NnN$ such that $\Phi$ is one-to-one
on $\bbB \setminus \zzZ$.
\end{pro}
\begin{proof}
We only need to show the `if' part of the proposition. Assume $\bbB$ is a standard set such that
the restriction of $\Phi$ to $\bbB$ is one-to-one. We may and do assume that
$(\bbB,\Mm\bigr|_{\bbB})$ is a standard Borel space. Let $\mu$ be a finite measure
on $\Mm\bigr|_{\bbB}$ for which $\NnN(\mu) = \NnN\bigr|_{\bbB}$. Since $\bbB$ and $\Ii(k)$ are
standard Borel spaces, we conclude that $\Phi(\bbB)$ is Borel and the restriction of $\Phi$
to $\bbB$ is a Borel isomorphism. Let $\psi\dd \Ii(k) \to \ssS(k)$ be as in \THM{cov}. Denote
by $\lambda$ the transport of $\mu$ under $\psi \circ \Phi$ (that is, $\lambda(\aaA) =
\mu(\Phi^{-1}(\psi^{-1}(\aaA)))$ for any $\aaA \in \Bb(\ssS(k))$). We infer from \THM{meas} that
there is a measurable domain $\ffF \subset \ssS(k)$ such that $\lambda(\ssS(k) \setminus \ffF) = 0$.
Now we put $\zzZ \df \bbB \setminus \Phi^{-1}(\psi^{-1}(\ffF)) \in \Mm$. Observe that $\mu(\zzZ) =
0$ and hence $\zzZ \in \NnN$. What is more, $\psi(\Phi(\bbB \setminus \zzZ)) \subset \ffF$ and thus
$\psi(\Phi(\bbB \setminus \zzZ))$ is a measurable domain. So, an application of \THM{cov} finishes
the proof.
\end{proof}

With the help of \PRO{cov} we can now describe a covering for $\IiI^f_I$.

\begin{pro}{full}
Let $\mmM = \{\mu_s\}_{s \in S}$ be a maximal collection of mutually singular probabilistic Borel
measures on $\Ii(k)$. Put $\xxX \df \Ii(k) \times S$,
\begin{gather*}
\Mm \df \{\bbB \subset \xxX|\quad \forall s \in S\dd\ \bbB^s \in \Bb(\Ii(k))\},\\
\NnN \df \{\zzZ \in \Mm|\quad \forall s \in S\dd\ \bbB^s \in \NnN(\mu_s)\}
\end{gather*}
where $\bbB^s = \{\ueX \in \Ii(k)\dd\ (\ueX,s) \in \bbB\}$ for each $s \in S$; and let $\Phi\dd \xxX
\to \Ii(k)$ be the projection onto the first coordinate. Then $(\xxX,\Mm,\NnN)$ is a multi-standard
measurable space with nullity and $\Phi$ is a covering such that $\IiI(\Phi) = \IiI^f_I$.
\end{pro}
\begin{proof}
It is clear that $(\xxX,\Mm,\NnN)$ is multi-standard and $\{\Ii(k) \times \{s\}\}_{s \in S}$ is
its standard base. There is also no difficulty in checking that $\Phi$ is measurable. We shall now
show that $\Phi$ is a covering. Let $\bbB \in \Mm$ be a standard set in $\xxX$. According
to \PRO{cov}, it suffices to show that there is $\zzZ \in \NnN$ such that the restriction of $\Phi$
to $\bbB \setminus \zzZ$ is one-to-one. To this end, first observe that are only countably many
indices $s \in S$ for which $\bbB^s \notin \NnN(\mu_s)$. Denote by $S_0$ the set of all such $s \in
S$. Since the measures $\mu_s\ (s \in S_0)$ are mutually singular, we conclude from the countability
of $S_0$ that there is a collection $\{\Dd_s\}_{s \in S_0}$ of pairwise disjoint Borel subsets
of $\Ii(k)$ such that $\mu_s(\Dd_s) = 1$ for each $s \in S_0$. Then $\zzZ \df \bbB \setminus
\bigcup_{s \in S_0} (\Dd_s \times \{s\})$ is a member of $\NnN$ we searched for.\par
It remains to check that $\IiI(\Phi) = \IiI^f_I$. Inclusion `$\subset$' follows from item (D)
of \THM{prime}. The reverse inclusion may simply be deduced from the proof of Proposition~5.6.13
in \cite{pn1}. Below we repeat its main arguments. Denote by $\ueJ^f_I$ the so-called \textit{unit}
of $\IiI^f_I$ (to read more on unities of ideals, consult \S3.5 in \cite{pn1}) and put $\ueT \df
\bigoplus^{\NnN}_{x\in\xxX} \Phi(x)$. Then there exists a $k$-tuple of operators whose unitary
equivalence class $\ueX$ satisfies: $\ueJ^f_I = \ueT \oplus \ueX$ and $\ueT \disj \ueX$.
As in the proof mentioned above, it suffices to show that $\ueX$ is trivial (i.e. its
representatives act on zero-dimensional Hilbert spaces). Assuming, on the contrary, that $\ueX$ is
nontrivial, we conclude that there is a Borel probabilistic measure $\mu$ on $\IiI(k)$ such that
$\ueX_0 \disj \ueT$ and $\ueX_0 \in \IiI^f_I$ where $\ueX_0 = \bigoplus^{\NnN(\mu)}_{\ueT\in\Ii(k)}
\ueT$. (The fact that $\mu$ is defined on $\Ii(k)$ follows from the latter property of $\ueX_0$.)
Now the unitary disjointness of $\ueX_0$ and $\ueT$ implies that $\mu$ is singular to each
of the measures $\mu_s\ (s \in S)$, which contradicts the maximality of the collection $\mmM$.
\end{proof}

As a consequence of \PRO{full} and \THM{prime} we obtain

\begin{cor}{full}
Let $\mmM$, $(\xxX,\Mm,\NnN)$ and $\Phi$ be as in \PRO{full}. For every $\ueX \in \IiI^f_I$ there
exists a \textup{(}unique up to $\NnN$-almost everywhere equality\textup{)} measurable
cardinal-valued function $m_{\ueX}\dd \xxX \to \Card$ such that $\ueX = \bigoplus^{\NnN}_{x\in\xxX}
m_{\ueX}(x) \odot \Phi(x)$.
\end{cor}

We are now ready to turn to the main subject of the section.

\begin{dfn}{bor}
Let $\mmM$, $(\xxX,\Mm,\NnN)$ and $\Phi$ be as above. For any Borel subset $\Ff$ of $\Ii(k)$ let
$\IiI[\Ff]$ be the class of all unitary equivalence classes $\ueX$ for which $m_{\ueX}(x) = 0$ for
$\NnN$-almost all $x \in \Phi^{-1}(\Ff)$.
\end{dfn}

Since coverings are unique (up to isomorphism), the definition of $\IiI[\Ff]$ introduced below is
independent of the choice of the collection $\mmM$ of mutually singular measures.\par
Our first claim is

\begin{lem}{F-id}
For each Borel set $\Ff \subset \Ii(k)$, the class $\IiI[\Ff]$ is an ideal of unitary equivalence
classes of $k$-tuples.
\end{lem}
\begin{proof}
Let $\mmM$, $(\xxX,\Mm,\NnN)$ and $\Phi$ be as usual. Let $j_{\Ff}\dd \xxX \to \{0,1\} (\subset
\Card)$ be the characteristic function of $\Phi^{-1}(\Ff)$ and put $\ueJ[\Ff] \df
\bigoplus^{\NnN}_{x\in\xxX} j_{\Ff}(x) \odot \Phi(x)$. Then for an arbitrary measurable function
$m\dd \xxX \to \Card$ we have:
\begin{equation*}
\contS[\!]{x\in\xxX}{\NnN} m(x) \odot \Phi(x) \in \IiI[\Ff] \iff m(x) \cdot j_{\Ff}(x) = m(x)
\textup{ for $\NnN$-almost all } x \in \xxX.
\end{equation*}
According to Theorem~5.6.14 in \cite{pn1}, the above is equivalent to:
\begin{equation*}
\ueX \in \IiI[\Ff] \iff \ueX \ll \ueJ[\Ff]
\end{equation*}
(where $\ueA \ll \ueB$ means that $\alpha \odot \ueB = \ueA \oplus \ueS$ for some cardinal $\alpha$
and a unitary equivalence class $\ueS$). Finally, again from \cite{pn1} it follows that the class
$\{\ueX\dd\ \ueX \ll \ueJ[\Ff]\}$ is an ideal.
\end{proof}

Mimicing the proof of \PRO{cov}, one shows that

\begin{pro}{F}
Let $\Ff$ be a Borel subset of $\Ii(k)$ and let $\mmM_{\Ff} = \{\mu_s\}_{s \in S}$ be a maximal
collection of mutually singular probabilistic Borel measures on $\Ii(k)$ each of which is supported
on $\Ff$. Put $\xxX_{\Ff} \df \Ff \times S$,
\begin{gather*}
\Mm_{\Ff} \df \{\bbB \subset \xxX_{\Ff}|\quad \forall s \in S\dd\ \bbB^s \in \Bb(\Ii(k))\},\\
\NnN_{\Ff} \df \{\zzZ \in \Mm_{\Ff}|\quad \forall s \in S\dd\ \bbB^s \in \NnN(\mu_s)\}
\end{gather*}
and let $\Phi_{\Ff}\dd \xxX_{\Ff} \to \Ff$ be the projection onto the first coordinate. Then
$(\xxX_{\Ff},\Mm_{\Ff},\NnN_{\Ff})$ is a multi-standard measurable space with nullity and
$\Phi_{\Ff}$ is a covering such that $\IiI(\Phi_{\Ff}) = \IiI[\Ff]$.
\end{pro}

The details of the proof are left to the reader.\par
Below we give a few illustrative examples of ideals of the form $\IiI[\Ff]$. Recall that $k$ denotes
the length of tuples of operators. Below we (naturally) identify subsets of $\CCC$ with subsets
of $\Ii_1(1)$.
\begin{itemize}
\item $\IiI[\CCC]$ is the ideal of single normal operators;
\item $\IiI[\RRR]$ is the ideal of single selfadjoint operators;
\item $\IiI[\TTT]$ is the ideal of single unitary operators ($\TTT$ is the unit circle);
\item $\IiI[\Ii_1(k)]$ is the ideal of $k$-tuples of commuting normal operators;
\item $\IiI[\Ii(k)] = \IiI^f_I$.
\end{itemize}

The issue we want to focus on here is the problem of p-isomorphicity of ideals. To this end,
we introduce

\begin{dfn}{p-iso}
Let $\IiI$ and $\JjJ$ be two finite type~I ideals of unitary equivalence classes of, respectively,
$k$-tuples and $\ell$-tuples of operators. $\IiI$ and $\JjJ$ are said to be \textit{p-isomorphic},
if there exists a bijection $\Psi\dd \IiI \to \JjJ$ such that:
\begin{enumerate}[(p{I}1)]
\item for every nonempty collection $\{\ueX_s\}_{s \in S}$ of members of $\IiI$,
 $\Psi(\bigoplus_{s \in S} \ueX_s) = \bigoplus_{s \in S} \Psi(\ueX_s)$;
\item for any $\ueX \in \IiI$, $\Psi(\ueX) \in \Ii(\ell) \iff \ueX \in \Ii(k)$;
\item there exists a covering $\Phi\dd \xxX \to \Ii(k)$ with $\IiI(\Phi) = \IiI$ such that:
 \begin{itemize}
 \item $\Psi \circ \Phi\dd \xxX \to \Ii(\ell)$ is a covering with $\IiI(\Psi \circ \Phi) = \JjJ$;
 \item for every cardinal-valued measurable function $m\dd \xxX \to \Card$,
  \begin{equation*}
  \Psi\Bigl(\contS[\!]{x\in\xxX}{\NnN} m(x) \odot \Phi(x)\Bigr) = \contS[\!]{x\in\xxX}{\NnN} m(x)
  \odot (\Psi \circ \Phi)(x).
  \end{equation*}
 \end{itemize}
\end{enumerate}
In the above situation, $\Psi$ is said to be a \textit{p-isomorphism}.
\end{dfn}

The concept of p-isomorphisms was introduced in \cite{pn1} (see \S6.3.3 there). Two ideals which are
p-isomorphic have similar structures induced by the direct sum operation and, in a sense, the same
Borel and spectral complexities (here `spectral' refers to properties expressible by means
of coverings). Roughly speaking, whatever can be said about one of them in terms of direct sums and
coverings, it has its natural counterpart for the other ideal.\par
Our main result on p-isomorphic ideals is formulated below.

\begin{thm}{p-iso}
If $\Ff \subset \Ii(k)$ and $\Gg \subset \Ii(\ell)$ are two Borel sets of the same cardinality, then
the ideals $\IiI[\Ff]$ and $\IiI[\Gg]$ \textup{(}of unitary equivalence classes of $k$-tuples and
$\ell$-tuples, respectively\textup{)} are p-isomorphic. In particular, $\IiI[\Ff]$ is p-isomorphic
to the ideal $\IiI[\TTT]$ of all single unitary operators provided $\Ff$ is uncountable.
\end{thm}
\begin{proof}
Since $\Ff$ and $\Gg$ are standard Borel spaces of the same cardinality, there exists a Borel
isomorphism $\psi\dd \Ff \to \Gg$ between them. Let $\mmM \df \mmM_{\Ff} = \{\mu_s\}_{s \in S}$,
$(\xxX,\Mm,\NnN) \df (\xxX_{\Ff},\Mm_{\Ff},\NnN_{\Ff})$ and $\Phi \df \Phi_{\Ff}\dd \xxX \to \Ff$ be
as in \PRO{F} (applied for $\Ff$). For each $s \in S$ let $\nu_s\dd \Bb(\Ii(\ell)) \to [0,1]$ be
the transport of $\mu_s$ under $\psi$. Then $\mmM_{\Gg} \df \{\nu_s\}_{s \in S}$ is a maximal
collection of mutually singular Borel probabilistic measures on $\Ii(\ell)$ which are supported
on $\Gg$, since $\psi$ is a Borel isomorphism. Now starting with $\mmM_{\Gg}$, we build
$(\xxX_{\Gg},\Mm_{\Gg},\NnN_{\Gg})$ and $\Phi_{\Gg}\dd \xxX_{\Gg} \to \Gg$ as in \PRO{F} (with $\Gg$
in place of $\Ff$). Then $\Phi_{\Gg}$ is a covering such that $\IiI(\Phi_{\Gg}) = \IiI[\Gg]$. Let
$\Psi_0\dd \xxX \to \xxX_{\Gg}$ be given by $\Psi_0(\ueT,s) = (\psi(\ueT),s)$. It is easy to check
that $\Psi_0$ is a bijection such that $\Mm_{\Gg} = \{\Psi_0(\bbB)\dd\ \bbB \in \Mm\}$ and
$\NnN_{\Gg} = \{\Psi_0(\zzZ)\dd\ \zzZ \in \NnN\}$. So, $\Psi_0$ is an isomorphism between measurable
spaces with nullities. We therefore conclude that $\Phi_{\Gg} \circ \Phi_0$ is a covering for
$\IiI[\Gg]$. But $\Phi_{\Gg} \circ \Psi_0 = \psi \circ \Phi$. We now define $\Psi\dd \IiI[\Ff] \to
\IiI[\Gg]$ by the rule:
\begin{equation*}
\Psi\Bigl(\contS[\!]{x\in\xxX}{\NnN} m(x) \odot \Phi(x)\Bigr) \df \contS[\!]{x\in\xxX}{\NnN} m(x)
\odot \psi(\Phi(x))
\end{equation*}
where $m\dd \xxX \to \Card$ is a cardinal-valued measurable function. Item (C) of \THM{prime}
implies that: the definition of $\Psi$ is complete and correct; $\Psi$ is a bijection; and
$\Psi(\ueX \oplus \ueY) = \Psi(\ueX) \oplus \Psi(\ueY)$ for any $\ueX, \ueY \in \IiI[\Ff]$. So,
it follows from Theorem~6.1.4 in \cite{pn1} that $\Psi$ satisfies condition (pI1). Moreover,
$\Psi(\ueX) = \psi(\ueX)$ for each $\ueX \in \Ff$, which yields (pI2) and (pI3). Now the note that
each uncountable Borel set in $\Ii(k)$ has the same cardinality as $\TTT$ completes the proof.
\end{proof}

We conclude the paper with a note concerning multiplicity theory for tuples of operators affiliated
with finite type~I von Neumann algebras.

\begin{rem}{multiplicity}
Based on the results obtained in the paper, one may easily extend, in a transparent way,
multiplicity theory for normal operators to arbitrary (finite tuples of) operators which are
affiliated with finite type~I von Neumann algebras. This idea may be realised as follows. Let $\mmM
= \{\mu_s\}_{s \in S}$, $(\xxX,\Mm,\NnN)$ and $\Phi\dd \xxX \to \Ii(k)$ be as in \PRO{full}. For
$\ueX \in \IiI^f_I$ we use $m_{\ueX}$ to denote the multiplicity function of $\ueX$ (relative
to $\Phi$). Now if $\tTT = (T_1,\ldots,T_k)$ acts in a separable Hilbert space and its unitary
equivalence class $\ueT$ belongs to $\IiI^f_I$, then there is a standard set $\bbB \subset \xxX$
such that $m_{\ueT}(x) = 0$ for $\NnN$-almost all $x \in \xxX \setminus \bbB$ and $m_{\ueT}(x)
\leqsl \aleph_0$ for each $x \in \bbB$ (see e.g. item (h) in Theorem~5.6.14 in \cite{pn1}). These
imply that we may think of $m_{\ueT}$ as of a function with values in $\{0,1,2,\ldots,\infty\}$.
Moreover, since $\bbB$ is standard, the set $S_0 = \{s \in S\dd\ \bbB^s \notin \NnN(\mu_s)\}$ is
countable (finite or not), say $S_0 = \{s_n\dd\ 1 \leqsl n < N\}$ (where $N$ is finite or not).
Further, we can find a collection $\{\Ss_n\dd\ 1 \leqsl n < N\}$ of pairwise disjoint Borel subsets
of $\Ii(k)$ such that $\mu_{s_n}(\Ss_n) = 1$ whenever $n < N$. Now define a measure $\mu_{\tTT}\dd
\Bb(\Ii(k)) \to [0,1]$ as $\mu_{\tTT} \df \sum_{n < N} \frac{1}{2^n} \mu_{s_n}$ and a function
$m_{\tTT}\dd \Ii(k) \to \{0,1,2,\ldots,\infty\}$ by the rules: $m_{\tTT}(\ueX) = m_{\ueT}(\ueX,s_n)$
for $\ueT \in \Ss_n$ (with $n < N$) and $m_{\ueT}(\ueX) = 0$ for $\ueX \notin \bigcup_{n<N} \Ss_n$.
In this way to each $k$-tuple $\tTT$ we have assigned a finite Borel measure $\mu_{\tTT}$
on $\Ii(k)$ and a measurable function $m_{\tTT}\dd \Ii(k) \to \{0,1,2,\ldots,\infty\}$. One may show
that two $k$-tuples $\tTT$ and $\sSS$ (acting in separable Hilbert spaces) are unitarily equivalent
iff the measures $\mu_{\tTT}$ and $\mu_{\sSS}$ are mutually absolutely continuous and
$m_{\tTT}(\ueX) = m_{\sSS}(\ueX)$ for $\mu_{\tTT}$-almost all $\ueX \in \Ii(k)$. The details are
left to the reader.
\end{rem}


\begin{thebibliography}{20}

\bibitem{b-k} H. Becker and A.S. Kechris,
 \textit{The Descriptive Set Theory of Polish Group Actions} 
 (London Math. Soc. Lecture Note Series, vol. 232),
 University Press, Cambridge, 1996.

\bibitem{bfh} A. Brown, C.-K. Fong, D.W. Hadwin,
 \textit{Parts of operators on Hilbert space},
 Illinois J. Math. \textbf{22} (1978), 306--314.

\bibitem{cas} C. Castaing,
 \textit{Quelques probl\`{e}mes de mesurabilit\'{e} li\'{e}es \`{a} la th\'{e}orie de la commande},
 C. R. Paris \textbf{262} (1966), 409--411.

\bibitem{con} J.B. Conway,
 \textit{A Course in Functional Analysis}
 (Graduate Texts in Mathematics, vol. 96),
 Springer, New York, 1990.

\bibitem{dix} J. Dixmier,
 \textit{$C^*$-algebras},
 North-Holland Publ. Co., Amsterdam, 1977.

\bibitem{ern} J. Ernest,
 \textit{Charting the operator terrain},
 Mem. Amer. Math. Soc. \textbf{171} (1976), 207 pp.

\bibitem{kec} A.S. Kechris,
 \textit{Classical Descriptive Set Theory}
 (Graduate Texts in Mathematics, Volume 156),
 Springer-Verlag, New York, 1995.

\bibitem{k-m} K. Kuratowski and A. Mostowski,
 \textit{Set Theory with an Introduction to Descriptive Set Theory},
 PWN -- Polish Scientific Publishers, Warszawa, 1976.

\bibitem{lon} R. Longo,
 \textit{Solution of the factorial Stone-Weierstrass conjecture. An application of the theory
  of standard split $W^*$-inclusions},
 Invent. Math. \textbf{76} (1984), 145--155.

\bibitem{pn1} P. Niemiec,
 \textit{Unitary equivalence and decompositions of finite systems of closed densely defined
  operators in Hilbert spaces},
 Dissertationes Math. (Rozprawy Mat.) \textbf{482} (2012), 1--106.

\bibitem{pn2} P. Niemiec,
 \textit{Elementary approach to homogeneous $C^*$-algebras},
 to appear (\texttt{http://arxiv.org/abs/1203.0857}).

\bibitem{pop} S. Popa,
 \textit{Semiregular maximal abelian $*$-subalgebras and the solution to the factor state
  Stone-Weierstrass problem},
 Invent. Math. \textbf{76} (1984), 157--161.

\bibitem{neu} J. von Neumann,
 \textit{On Rings of Operators. Reduction Theory},
 Ann. Math. \textbf{50} (1949), 401--485.

\end{thebibliography}
\end{document}